\documentclass[12pt]{amsart}
\usepackage{amsmath,amssymb,amsthm}
\usepackage[utf8]{inputenc}
\usepackage[all]{xy}

\def\ol#1{{\overline{#1}}}
\def\we{\wedge}
\def\wt{\widetilde}
\newcommand{\pt}{{\partial }}
\newcommand{\cA}{{\mathcal A} }
\newcommand{\cO}{{\mathcal O} }

\begin{document}

\title{Curvature of higher direct image sheaves}

\dedicatory{Dedicated to Professor Yujiro Kawamata\\on the occasion of his 60th birthday}

\author{Thomas Geiger}

\address{Fachbereich Mathematik und Informatik,
         Philipps-Universit\"at Marburg,
         Lahnberge,
         Hans-Meerwein-Stra\ss e,
         D-35032
         Marburg,
         Germany}

\email{geigert@mathematik.uni-marburg.de}

% Define an author
\author{Georg Schumacher}

\address{Fachbereich Mathematik und Informatik,
         Philipps-Universit\"at Marburg,
         Lahnberge,
         Hans-Meerwein-Stra\ss e,
         D-35032
         Marburg,
         Germany}

\email{schumac@mathematik.uni-marburg.de}

% Define theorem environments
\newtheorem{theorem}{Theorem}
\newtheorem{definition}{Definition}
\newtheorem{lemma}{Lemma}
\newtheorem{remark}{Remark}
\newtheorem{proposition}{Proposition}
\newtheorem{corollary}{Corollary}
\newtheorem{claim}{Claim}
\newtheorem{observation}{Observation}

% Define keywords

% Define subject classification
% remarks of Georg Schumacher:
% 32L10   	Sheaves and cohomology of sections of holomorphic vector bundles, general results
% 14D20   	Algebraic moduli problems, moduli of vector bundles

\begin{abstract}
Given a family $(F,h) \to X \times S$ of Hermite-Einstein bundles on a compact Kähler manifold $(X,g)$ we consider the higher direct image sheaves $R^q p_* \mathcal{O}(F)$ on $S$, where $p: X \times S \to S$ is the projection. On the complement of an analytic subset these sheaves are locally free and carry a natural metric, induced by the $L_2$ inner product of harmonic forms on the fibers. We compute the curvature of this metric which has a simpler form for families with fixed determinant and families of endomorphism bundles. Furthermore, we discuss the metric for moduli spaces of stable vector bundles.
\end{abstract}

\maketitle

% ********************************************************************************************************************************************************
% Introduction
% ********************************************************************************************************************************************************
\section{Introduction}\label{sec:intro}
Given a compact Kähler manifold $(X,g)$ and a complex space $S$, a family of Hermite-Einstein bundles on $X$ parameterized by $S$ is a holomorphic vector bundle $F \to X \times S$ with a hermitian metric $h$ such that the restriction $h_s$ of $h$ to $F_s = F|X\times\{s\}$ is a Hermite-Einstein metric for every $s \in S$. If $p: X \times S \to S$ denotes the projection, the higher direct image sheaf $R^q p_* \mathcal{O}(F)$ for $q \in \mathbb{N}$ is locally free outside some analytic set of $S$ and carries a natural hermitian metric, fiberwise induced by the $L_2$ inner product of harmonic representatives.

If $\rho_k\in \cA^{0,1}(X,\mathrm{End}(F_s))$ denotes the harmonic representatives of  Kodaira-Spencer classes, then for all $q$ there exist natural mappings
\begin{eqnarray*}
  \rho_k \cup \textvisiblespace &:& \cA^{0,q}(X, F_s) \to \cA^{0,q+1}(X, F_s)\\
  \rho^*_k \cap \textvisiblespace &:& \cA^{0,q}(X, F_s) \to \cA^{0,q-1}(X, F_s),
\end{eqnarray*}
where the second mappings are adjoint to the first.
Denote by $\xi_\sigma$ harmonic sections of $\cA^{0,q}(X, F_s)$. Then we denote the pointwise inner product of such sections (on a fiber $X\times \{s\}$) that is induced by $h_s$ and $g$ by $(\xi_\rho,\xi_\sigma)$ and the above $L_2$ inner product by
$$
\langle\xi_\rho,\xi_\sigma\rangle= \int_{X \times \{s\}} (\xi_\rho,\xi_\sigma)\frac{\omega^n}{n!},
$$
where $n$ is the dimension of $X$ and $\omega$ the Kähler form of $g$.
We assume that $S$ is smooth and compute the curvature of the natural metric on  $R^q p_* \mathcal{O}(F)$. Under the above assumptions we have:

\begin{theorem}\label{thm:maintheorem}
Let $(F,h)$ be a holomorphic, hermitian vector bundle on $X \times S$ such that all vector bundles $F_s=F|X\times\{s\}$, $s\in S$ are simple, and all metrics $h_s$ are Hermite-Einstein. Then the curvature tensor of the induced metric on $R^qp_*\cO(F)$  is given by:
\begin{align*}
  R_{\rho\ol\sigma k \overline{l}}
  &=
  \langle
    G
    (
        \rho_l^*
      \cap
        \xi_\rho
    )
    ,
      \rho_k^*
    \cap
      \xi_\sigma
  \rangle  \\
  &\hphantom{=}+
  \langle
    G
    \left(
      \sqrt{-1}
      \Lambda_g
      \left[
          \rho_k,
          \rho_l^*
      \right]
    \right)
      \xi_\rho
    ,
      \xi_\sigma
  \rangle  \\
  &\hphantom{=}-
  \langle
    G
    (
        \rho_k
      \cup
        \xi_\rho
    ),
      \rho_l
    \cup
      \xi_\sigma
  \rangle  \\
  &\hphantom{=}+
  \langle
    H(\rho_{k\overline{l}})\,\xi_\rho,
      \xi_\sigma
  \rangle .
\end{align*}
For $q=0$ the first, and for $q=\dim X$ the third summand vanishes.
\end{theorem}

The meaning of the function $H(\rho_{k \ol l})$ that depends only on the base parameter $s$ will be explained in Section~\ref{sec:curvature}.

The obvious key point about this formula is to express the curvature tensor in terms of intrinsically given data like harmonic Kodaira-Spencer classes. Technically this means the elimination of second order derivatives of the metric tensor when computing the curvature.

This result contains the case of direct image sheaves ($q=0$) which was solved in the work of To and Weng \cite{TW98}. Our main motivation is the study of the curvature of the Weil-Petersson metric on the moduli space of stable bundles in \cite{ST92} and also Berndtsson's positivity results for higher direct image sheaves \cite{Be09}, which was continued in the work of Berndtsson-P\u{a}un \cite{BP08} and Mourougane-Takayama \cite{MT08,MT09}.

At the end of section \ref{sec:curvature} we obtain the following results that are related to the Theorem~\ref{thm:maintheorem}. Namely, for families with fixed determinant we can rescale the Hermite-Einstein metrics locally and see that
$H(\rho_{k\overline{l}}|_{s_0})= 0$ so that the fourth summand in the curvature formula vanishes for such families. Furthermore, this fourth summand vanishes for induced families of endomorphism bundles without further assumptions.

This last case is especially interesting for applications to the moduli spaces of stable bundles. Let us consider the Kodaira-Spencer map
\[
  \rho_s: T_s S \longrightarrow H^1(X, \mathcal{O}(\mathrm{End}(F_s)))
\]
for $s \in S$, associated to a family $(F,h) \to X \times S$ of hermitian bundles on $X$. For any complex tangent vector  $v$ of $S$ at $s$  the Kodaira-Spencer map in terms of Dolbeault cohomology is given by
\begin{equation}\label{intro:overhaus}
  \rho_s(v|_s) =
  \left[ -( v \,\lrcorner\, \Omega^h)|_s \right]
\end{equation}
where $\Omega^h$ denotes the curvature form of $h$, and $|_s$ always denotes a restriction to $X\times \{s\}$. We use this notation unless obvious. Furthermore, the Hermite-Einstein condition implies that
$-( v \,\lrcorner\, \Omega^h)|_s
$
is {\em harmonic} \cite{Ov92,ST92}.  In particular, one can read off this formula that the harmonic Kodaira-Spencer tensors depend in a differentiable way on the parameter. Altogether, we have a close relationship between the variation of the holomorphic structure on a complex vector bundle and the metric structure that is induced by $h$ on $F$ over $X\times S$.

For complex tangent vectors $v$ and $w$ on $S$ at $s$ the induced natural inner product reads
\[
  \langle
    v, w
  \rangle _{\mathrm{WP}}(s)
  =
    \langle
      \rho_s(v) , \rho_s(w)
    \rangle(s)
  =
    \langle
      \left.
        (v \,\lrcorner\, \Omega^h)
      \right|_s,
      \left.
        (w \,\lrcorner\, \Omega^h)
      \right|_s
    \rangle
\]
for $s \in S$. For effective families we get a hermitian metric on $S$ which is known to be Kähler. Moreover, this construction is functorial, i.e.\ compatible with base change and, as a consequence, descends to the corresponding moduli space of stable bundles. Due to its analogy with the Weil-Petersson metric on the Teichmüller space of Riemann surfaces, it is also called {\em Weil-Petersson metric}. In \cite{ST92} the curvature tensor of this metric was computed. The result follows as a special case of our main theorem. Under our assumptions the fibers of $R^1 p_* \mathcal{O}(\mathrm{End}(F))$ are $H^1(X, \mathcal{O}(\mathrm{End}(F_s)))$. So it is sufficient to apply the Theorem to the induced family of endomorphism bundles in the case $q=1$ and use the simplification for families of endomorphism bundles mentioned above.

{\bf Acknowledgement.} {\em This article is dedicated to Professor Yujiro Kawamata on the occasion of his birthday with respect and admiration.}
% ********************************************************************************************************************************************************
% The natural L2-metric
% ********************************************************************************************************************************************************
\section{The natural $L_2$ metric}\label{sec:l2metric}
By a theorem of Grauert, for any family of holomorphic vector bundles $(F,h) \to X \times S$ parameterized by a reduced space $S$, the dimension $h^q(X, \mathcal O(F_s))$ is constant on the complement of a certain analytic subset of $S$. Furthermore, the sheaf $R^q p_* \mathcal{O}(F)$ is locally free over this complement and the natural morphism
\begin{equation}\label{l2metric:fiberofimagesheave}
  (R^q p_* \mathcal{O}(F))_s
  \otimes_{\mathcal{O}_{S,s}}
  \mathbb{C}(s)
  \stackrel{\sim}{\longrightarrow}
  H^q(X, \mathcal{O}(F_s))
\end{equation}
is an isomorphism by the base-change theorem (cf.\ \cite{BS76}). Given the Kähler structure on $X$ and the hermitian metric $h$ on $F$ over $X\times S$ we identify the cohomology groups with the spaces $\mathcal{H}^{0,q}(X, F_s)$ of forms that are harmonic with respect to $g$ and $h_s$. On these we have a natural inner product, given by
\[
  \langle \mu, \eta \rangle
  =
  \int_X
    (\mu, \eta)
    \,\frac{\omega^n}{n!}.
\]
Since our arguments will be local with respect to the parameter space $S$, we may assume that $S$ is Stein and identify sections of direct image sheaves with cohomology classes in $H^q(X\times S, \mathcal{O}(F))$. On one hand, these are represented by {\em $\ol{\partial}$-closed forms} over $X\times S$, on the other hand, by {\em families of harmonic forms} in $\mathcal H^{0,q}(X,F_s)$. The following lemma (cf.\ \cite[Lemma 2]{Sch12}) shows that both properties can be achieved simultaneously, a fact that is necessary for later computations.

\begin{lemma}\label{lemma:correctionforsections}
Let $(F,h) \to X \times S$ be a family of hermitian bundles on a compact Kähler manifold $(X,g)$ parameterized by a complex manifold $S$ and assume that $R^q p_*\mathcal O(F)$ is locally free. If
$
  \phi\in\mathcal{A}^{0,q}(X\times S, F)
$
is $\ol{\partial}$-closed and $s_0 \in S$ any point, then there exists a form $\chi\in \mathcal{A}^{0,q-1}(X\times V,F) $ on some open neighborhood $V \subset S$ of $s_0$ such that
\[
    (\phi + \ol{\partial} \chi)|_s   =   H(\phi|_s)
\]
for every $s \in V$, where $H$ is the harmonic projection for the fiber $F_s$. In particular, any class in $H^q(X\times S, \mathcal O(F))$ can be represented by a $\ol{\partial}$-closed form, whose restrictions to all fibers of $s\in V$ are harmonic.
\end{lemma}
\begin{proof}
First, there exist relative $(0,q)$ forms
$
  \psi_G^\prime
$
and
$
  \psi_H^\prime
$
along the fibers $F_s \to X$ with $\psi_G^\prime|_s =  G(\phi|_s)$ as well as
$\psi_H^\prime|_s = H(\phi|_s)$ for every $s \in S$, where $G$ denotes the Green operator on $F_s$-valued forms. In fact, this is a consequence of our assumption and follows  from the fundamental theorem in \cite{KS58}, proved  in \cite[Theorem 7]{KS60}. On some open neighborhood $V$ of $s_0$ the relative form
$
  \psi_H^\prime
$
as well as
$
  \ol{\partial}^*_{\mathrm{rel}}
  \psi_G^\prime
$,
where
$
  \ol{\partial}^*_{\mathrm{rel}}
$
is the relative $\ol{\partial}^*$ operator, are induced by forms
$
  \psi_H
  \in
  \mathcal{A}^{0,q}(X \times V, F)
$
and
$
  \psi_G
  \in
  \mathcal{A}^{0,q-1}(X \times V, F)
$,
respectively. Initially the extensions of the relative forms can be taken locally, and then glued together by a partition of unity.  We get $\psi_G|_s =  \ol{\partial}^* G(\phi|_s)$ and also $\psi_H|_s = H(\phi|_s)$ for every $s \in V$. Using $\ol{\partial} \phi = 0$ we obtain
\[
    \phi|_s
  =
  H(\phi|_s)+\ol{\partial}\ol{\partial}^*G(\phi|_s)
\]
and, as a consequence, by choosing $\chi := - \psi_G$ we finally have
\[
    (\phi + \ol{\partial}\chi)|_s
  =
    \phi|_s  -
  \ol{\partial}(\psi_G|_s) =  H(\phi|_s).
\]
\end{proof}

% ********************************************************************************************************************************************************
% Computation of the curvature
% ********************************************************************************************************************************************************
\section{Computation of the curvature}\label{sec:curvature}
Again we will always assume that all $R^qp_*\mathcal O(F)$ are locally free. Although some statements are possible for reduced base spaces, we will also assume that $S$ is smooth. Since  curvature computations are local with respect to the base, we can assume without loss of generality that $S$ is Stein with local coordinates $(s^1,\ldots,s^m)$ and replace the space $S$ by a neighborhood of a given point $s_0\in S$, if necessary.

We will use holomorphic coordinates $(z^1,\ldots,z^n)$ for the Kähler manifold $X$ and denote the Kähler form by
$$
\omega=\frac{\sqrt{-1}}{2}g_{\alpha\ol \beta} dz^\alpha\we dz^{\ol{\beta}}.
$$
% We will use the summation convention everywhere.
Greek indices refer to the $z$-coordinates on $X$ (or to sections of the given vector bundle), whereas Latin indices are reserved for $s$-coordinates.

Let
$
  \Xi_1, \ldots, \Xi_R
  \in R^q p_* \mathcal{O}(F)(S)
$
define a holomorphic frame so that the inner product is given by $H_{\tau\ol\sigma}$. Then the curvature form $\Omega \in \mathcal A^{1,1}(X \times S,{\rm End}(R^qp_*\mathcal O(F)))$ is equal to
$$
\Omega = R^{\;\tau}_{\rho\; k \ol l} ds^k\wedge ds^{\ol l},
$$
in $s$-coordinates, and we have
$$
R_{\rho\ol\sigma k \ol l} = H_{\tau\ol\sigma}R^{\;\tau}_{\rho\; k \ol l}.
$$
Furthermore, we will work in normal coordinates at a given point $s_0$. After replacing $S$ by a neighborhood of $s_0$, if necessary, we can assume that
$$
H_{\tau \ol \sigma}(s_0)= \delta^\sigma_\tau \text{ and } \frac{\pt}{\pt s^k}H_{\rho\ol \sigma}(s_0)=0 \text{ for all } k,\sigma, \rho.
$$
We apply Lemma~\ref{lemma:correctionforsections} which provides representatives
$$
  \xi_1, \ldots, \xi_R
  \in
  \mathcal{A}^{0,q}(X \times S,F)
$$
of the Dolbeault cohomology classes of $\Xi_\rho$, such that the restrictions $\xi_\rho|_s$ are {\em harmonic}. We note
\begin{equation}\label{curvature:hrhosigma}
  H_{\rho\overline{\sigma}}(s)
  =
  \langle
    \Xi_\rho, \Xi_\sigma
  \rangle (s)
  =
  \langle
    \xi_\rho, \xi_\sigma
  \rangle (s).
\end{equation}
With respect to the given normal coordinates, we get
\begin{equation}\label{curvature:curvature}
R_{\rho\ol \sigma k \ol l}(s_0)= - \pt_{\ol l}\pt_k H_{\rho \ol \sigma}(s_0)   = - \pt_{\ol l}\pt_k \langle \xi_\rho, \xi_\sigma\rangle(s_0)
\end{equation}
for the curvature on the direct image sheaves.

For computational reasons we equip $S$ with the flat metric with respect to the coordinates $s$ and $X \times S$ with the product metric. This convention is not essential, but we are in a position to use covariant derivatives for tensors with values in $F$ over the total space $X\times S$.

We will use the following notation for differentiable sections $\chi$ of $F$ over $X\times S$.
The connection form of $h$ is
$$
\theta^h = \pt h \cdot h^{-1}
$$
so that
$$
\nabla_\alpha \chi = \pt_\alpha \chi + \theta^h_\alpha\circ \chi, \text{ and } \nabla_k \chi = \pt_k \chi + \theta^h_k\circ \chi \text{ resp.}
$$
(and $\nabla_\ol\beta=\pt_\ol \beta$, $\nabla_\ol l = \pt_\ol l$).
Now
\begin{equation}\label{eq:curvh}
\nabla_\ol{\beta}\nabla_\alpha \chi = \nabla_\alpha\nabla_\ol{\beta}\chi  - R^h_{\alpha \ol{\beta}} \circ \chi
\end{equation}
where $R^h_{\alpha \ol{\beta}}= -\nabla_{\ol{\beta}}\theta^h_\alpha$ etc.\ denote the components of the curvature tensor of $h$ with values in ${\rm End}(F)$ over $X \times S$. For the components in base direction the analogous equations hold.

For later use we note that for the metric induced by $h$ on ${\rm End}(F)$ and differentiable sections $\zeta$ the above formula reads
\begin{equation}\label{eq:curvhend}
\nabla_\ol{\beta}\nabla_\alpha \zeta = \nabla_\alpha\nabla_\ol{\beta}\zeta  - [R^h_{\alpha \ol{\beta}} , \zeta].
\end{equation}

In this section $\eta$ and $\mu$ will denote $F$-valued $(0,q)$-forms. The following construction will be essential:

\begin{lemma}\label{lemma:covarianttoexterior}
Let $\mu\in \cA^{0,q}(X\times S,F)$, $q >  0$ with $\ol{\partial} \mu = 0$. Then for every $1 \leq l \leq m$ there exists a form
$
  F_{\overline{l}}(\mu)
  \in
  \mathcal{A}^{0,q-1}(X \times S,F)
$
satisfying the following equation for every $s \in S$:
\[
  \ol{\partial}(F_{\overline{l}}(\mu)|_s)
  =(\nabla_{\overline{l}}\mu)|_s.
\]
In other words, the derivative in a (conjugate) parameter direction of a $\ol\pt$-closed form, restricted to a fiber is $\ol\pt$-exact.
\end{lemma}
\begin{proof}
We write
\begin{align*}%\label{lemma:covarianttoexterior:p0}
  \mu
  &=
  \frac{1}{q!}
  \sum_{\beta} \mu_{\overline{\beta}_1,\ldots,\overline{\beta}_q} dz^{\overline{\beta}_1} \wedge \ldots \wedge dz^{\overline{\beta}_q} \\\nonumber
  &\hphantom{=}+
  \frac{1}{(q-1)!}
  \sum_{\beta,k} \mu_{\overline{\beta}_1,\ldots,\overline{\beta}_{q-1},\overline{k}} dz^{\overline{\beta}_1} \wedge \ldots \wedge dz^{\overline{\beta}_{q-1}} \wedge ds^{\overline{k}} \\\nonumber
  &\hphantom{=}+
  \mbox{summands with more than one } \overline{d s} \mbox{-factor}.
\end{align*}
Because $\ol{\partial} \mu = 0$, for the fixed $l$ the factors
$
  dz^{\overline\beta_1} \wedge \ldots \wedge dz^{\overline\beta_q} \wedge ds^{\overline{\ell}}
$
yield:
\begin{align}\label{lemma:covarianttoexterior:p1}
  &\frac{1}{q!}
  \sum_\beta
    \partial_{\overline{l}}
    \mu_{\overline{\beta}_1,\ldots,\overline{\beta}_q}
    dz^{\overline{\beta}_1} \wedge \ldots \wedge dz^{\overline{\beta}_q}\\\nonumber
  &\hspace{1cm}=
  \frac{(-1)^{q+1}}{q!}
  \sum_\beta
  \sum_\nu
    (-1)^{\nu+1}
    \partial_{\overline{\beta_\nu}}
    \mu_{\overline{\beta}_1,\ldots,\widehat{\overline{\beta}_\nu},\ldots,\overline{\beta}_q,\overline{l}}
    dz^{\overline{\beta}_1} \wedge \ldots \wedge dz^{\overline{\beta}_q}.
\end{align}
After restricting to $X\times \{s\}$ the left hand side of \eqref{lemma:covarianttoexterior:p1} equals
$
  \left.
    \left(
      \nabla_{\overline{l}} \mu
    \right)
  \right|_s
$.
So
\begin{equation}\label{lemma:covarianttoexterior:p2}
  F_{\ol l}(\mu)
  :=
  \frac{(-1)^{q+1}}{(q-1)!}
  \sum_\beta
    \mu_{\overline{\beta}_1,\ldots,\overline{\beta}_{q-1},\overline{l}}
    dz^{\overline{\beta}_1} \wedge \ldots \wedge dz^{\overline{\beta}_{q-1}}
\end{equation}
has the desired properties.
\end{proof}

We will use the notion $\ol\pt^*$ only for differential forms {\em on the fibers} $X\times\{s\}$ with respect to $g$ and $h_s$ (with variable parameter). In this sense we have a first application:

\begin{corollary}\label{corollary:vanishingofproduct}
If $\ol{\partial}\eta = 0$ and
$
  \ol{\partial}^*(\mu|_s)  =  0
$
for some $s \in S$, then also
$
  \langle
    \mu, \nabla_{\overline{k}} \eta
  \rangle (s)
  =
  0
$.
\end{corollary}
\begin{proof}
In the case $q=0$ the result is clear because $\eta$ is holomorphic. In the case $q >0$ we can apply Lemma \ref{lemma:covarianttoexterior} to get on all fibers $X\times \{s\}$:
\[\langle \mu,\nabla_{\overline{k}} \eta \rangle(s)   =  \langle  \mu,\ol{\partial}(F_{\overline{k}}\eta)\rangle(s)
  =\langle \ol{\partial}^*\! \mu, F_{\overline{k}}(\eta)\rangle (s)  = 0.
\]
\end{proof}

We also obtain a formula for the second order derivatives:
\begin{corollary}\label{corollary:secondderivatives}
If we have $\ol{\partial}\eta = 0$ and
$
  \ol{\partial}^*(\mu|_s) = 0
$
on all fibers we get:
\[
  \partial_{\overline{l}} \partial_k \langle \mu, \eta\rangle   = \langle \nabla_k \mu, \nabla_l \eta\rangle    + \langle \nabla_{\overline{l}}\nabla_k \mu, \eta \rangle  .
\]
\end{corollary}

The values of
$
  \ol{\partial}^*\!(\nabla_k \mu|_s)
$
as well as
$
  \ol{\partial}(\nabla_k \mu|_s)
$
are of particular importance for the computation of the curvature. For the first type we find:

\begin{lemma}\label{lemma:dbarstarofnablamu}
If
$
  \ol{\partial}^*(\mu|_s)  =  0
$
for every $s \in S$ then also
\[
  \ol{\partial}^*(\nabla_k \mu|_s)  =  0
\]
holds for every $s \in S$.
\end{lemma}
\begin{proof}
Since the connection and curvature forms of $h$ in the direction of $X$ restricted to $X\times \{s\}$ equal the connection and curvature forms resp.\ of $h_s$, we have
\begin{gather*}
g^{\ol{\delta \gamma}}\nabla_\gamma(\nabla_k \mu_{\ol{\beta}_1,\ldots,\ol{\beta}_{q-1}|\ol{ \delta}}|_s )=  (g^{\ol{\delta \gamma}}\nabla_k\nabla_\gamma  \mu_{\ol{\beta}_1,\ldots,\ol{\beta}_{q-1}|\ol{ \delta}})|_s\\ \hspace{5cm} = (\nabla_k (g^{\ol{\delta} \gamma} \nabla_\gamma  \mu_{\ol{\beta}_1,\ldots,\ol{\beta}_{q-1}|\ol{ \delta}}) )|_s
\end{gather*}
which implies the claim.
\end{proof}

In order to compute expressions of the second type $\ol{\partial}(\nabla_k \mu|_s)$ we have to introduce some more notation. If $A$ is a $(p,q)$ form with values in some endomorphism bundle $\mathrm{End}(E)$, we agree to denote with $A \cup$ the operator on $(r,s)$ forms with values in $E$ which consists of the application of the endomorphism part and an exterior multiplication of the form parts. We will use $A^* \cap$ to denote the formal adjoint of $A \cup$.

In local coordinates we need the following case: Let $A=\sum A_\ol \delta dz^{\ol \delta}$ be an
${\rm End}(E)$-valued $(0,1)$-form, and $\mu=(1/q!) \sum \mu_{\ol\beta_1,\ldots,\ol \beta_q}dz^{\ol \beta_1}\wedge\ldots\wedge dz^{\ol \beta_q}$ an $E$-valued $(0,q)$-form. Then
\begin{equation}\label{eq:cup}
A\cup \mu = \frac{1}{(q+1)!} \sum A_{\ol\beta_0}(\mu_{\ol\beta_1,\ldots,\ol \beta_q})dz^{\ol \beta_0}\we\ldots\we dz^{\ol \beta_q},
\end{equation}
and for an $E$-valued $(0,q+1)$-form $\sigma$ (whose coefficients are already assumed to be skew-symmetric) we have
\begin{equation}\label{eq:cap}
A^* \cap \sigma =\sum g^{\ol \delta \gamma} A^*_{\gamma}( \sigma_{\ol \delta, \ol \beta_1,\ldots, \ol\beta_q}) dz^{\ol \beta_1}\we\ldots dz^{\ol \beta_q}.
\end{equation}

Furthermore, we will use the abbreviations
$$
  \rho_k
  :=
  - \partial_k \,\lrcorner\, \Omega^h
\text{\quad and \quad}
  \rho_{k\overline{l}}
  :=
  \partial_{\overline{l}} \,\lrcorner\, (\partial_k \,\lrcorner\, \Omega^h)
$$
as well as some obvious variations on these, where again $\Omega^h$ denotes the curvature form of $h$. Note that, in particular, we can use these to describe the harmonic representatives of Kodaira-Spencer classes, because \eqref{intro:overhaus} implies that in our notation the harmonic representative $\rho_s(\partial_k|_s)$ equals $\rho_k|_s$.
After this remark we obtain:

\begin{lemma}\label{lemma:dbarofnablamu}
Let $\mu\in\cA^{0,q}(X\times S,F)$ with $\ol{\partial}(\mu|_s)=0$ for every $s \in S$. Then
\[
  \ol{\partial}(\nabla_k \mu|_s)  =   \rho_k|_s \cup  \mu|_s
\]
for every $s \in S$.
\end{lemma}
\begin{proof}
We observe that
$$
\nabla_{\ol{\beta}_{0}}\nabla_k\, \mu_{\ol \beta_1, \ldots, \ol \beta_q} = \nabla_k\nabla_{\ol \beta_{0}} \mu_{\ol \beta_1, \ldots, \ol \beta_q} - R^h_{k\ol \beta_{0}}( \mu_{\ol \beta_1,\ldots,\ol\beta_q})
$$
and replace
\begin{equation}\label{eq:harmrep}
R^h_{k \ol{\beta}} dz^\ol\beta = \pt_k\, \lrcorner\, \Omega^h= - \rho_k.
\end{equation}
\end{proof}

We continue the calculation of the curvature. First we know from the construction of the good representatives $\xi_\rho$ of the Dolbeault classes that the restrictions
$
  \left.
    \xi_\rho
  \right|_s
$
are harmonic and the above identities are applicable.  In particular, Corollary \ref{corollary:secondderivatives} together with \eqref{curvature:curvature} implies
\[
  R_{\rho\sigma k \overline{l}}(s_0)
  =
  -
  \big(
    \langle
      \nabla_k \xi_\rho, \nabla_l \xi_\sigma
    \rangle
    (s_0)
    +
    \langle
      \nabla_{\overline{l}} \nabla_k \xi_\rho, \xi_\sigma
    \rangle
    (s_0)
  \big)
  =:
  -(S_1 + S_2).
\]
Using Corollary~\ref{corollary:vanishingofproduct} and the fact that we chose normal coordinates at $s_0$ we find
\[
  \langle \nabla_k \xi_\rho|_{s_0}, \xi_\sigma|_{s_0}\rangle  =\partial_k\langle\xi_\rho, \xi_\sigma\rangle(s_0) = \partial_k H_{\rho\overline{\sigma}}(s_0)  =  0
\]
i.e.\ the derivatives $\nabla_k\xi_\rho|_{s_0}$ are perpendicular to the space of harmonic forms $\mathcal H^{0,q}(X,F_{s_0})$ that is spanned by all forms $\xi_\sigma|_{s_0}$.
Hence, the harmonic projections $H(\nabla_k \xi_\rho|_{s_0})$ vanish so that by Lemma~\ref{lemma:dbarstarofnablamu}
\[  \nabla_k \xi_\rho|_{s_0} =\ol{\partial} G  \ol{\partial}^*(\nabla_k \xi_\rho|_{s_0})  + \ol{\partial}^*  G\ol{\partial}(\nabla_k \xi_\rho|_{s_0}) = \ol{\partial}^* G \ol{\partial}(   \nabla_k \xi_\rho|_{s_0}),
\]
where $G$ denotes the respective Green operator. As a consequence we establish for the first summand $S_1$
\begin{align*}
  S_1
  &=  \langle \nabla_k \xi_\rho, \nabla_l \xi_\sigma\rangle(s_0)  =  \big\langle   \ol{\partial}^* G \ol{\partial}(\nabla_k \xi_\rho),\nabla_l \xi_\sigma \big\rangle (s_0)  \\
  &=  \big\langle  G \ol{\partial}( \nabla_k  \xi_\rho),\ol{\partial} \nabla_l \xi_\sigma\big\rangle(s_0)  \\
    &= \big\langle G(\rho_k\cup\xi_\rho), \rho_l\cup\xi_\sigma \big\rangle(s_0)
\end{align*}
where we used Lemma \ref{lemma:dbarofnablamu} in the last step.

In order to calculate the second summand $S_2$ we first proceed with collecting further relations. For this purpose we note the following local formula
\begin{equation}\label{curvature:liebracketcovariantderiv}
  \nabla_k \nabla_{\overline{l}} \,\mu - \nabla_{\overline{l}} \nabla_k\, \mu
  =
  \rho_{k\overline{l}} \cup \mu
\end{equation}
which follows from \eqref{eq:curvh} and the definition of $\rho_{k\overline{l}}$. (It only depends on $h$ because of the choice of the Kähler metric on $X\times S$). We denote by $\Box$ the Laplacian (with non-negative eigenvalues) for $F_s$-valued $(0,q)$-forms.

\begin{lemma}\label{lemma:laplacerhokl}
In our situation on all fibers $X\times \{s\}$ the equation
\[\Box( \rho_{k\overline{l}}) =\sqrt{-1} \Lambda_g\left[\rho_k, \rho_{\overline{l}}
  \right]
\]
holds.
\end{lemma}
\begin{proof}
Let $R^h$ be the ${\rm End}(F)$-valued curvature tensor of $h$ over $X\times S$ as above.
The quantity $R^h_{k\ol l}$\/, when restricted to a fiber, has to be treated as a differentiable section of the endomorphism bundle of $F_s$. We compute
\begin{align*}
\Box(R^h_{k\ol l}) &= \ol\pt^* \ol\pt R^h_{k\ol l} = - g^{\ol \delta \gamma} \nabla_\gamma\nabla_\ol\delta R^h_{k \ol l}= - g^{\ol \delta \gamma} \nabla_\gamma\nabla_\ol l R^h_{k \ol \delta}\\
&=- g^{\ol \delta \gamma} [R^h_{\gamma \ol l}, R^h_{k \ol \delta}] - g^{\ol \delta \gamma} \nabla_\ol l\nabla_k  R^h_{\gamma \ol  \delta}.
\end{align*}
The very last term vanishes because of the Hermite-Einstein condition, since the degree of the bundles is constant. So the claim follows as above from \eqref{eq:harmrep}.
\end{proof}

We need two properties of the forms $F_{\overline{l}}(\mu)$ from Lemma \ref{lemma:covarianttoexterior}:

\begin{lemma}\label{lemma:property2ofCTEconstruction}
If $\ol{\partial} \mu = 0$ and $\ol{\partial}^*\!\mu|_s = 0$
for every $s \in S$, we obtain fiberwise for all $s$
\[
  \ol{\partial}^* \ol{\partial}(F_{\overline{l}}(\mu)) =\rho^*_{l}  \cap \mu.
\]
\end{lemma}
\begin{proof}
We have
\begin{align*}
\ol\pt^*\!\ol\pt(F_\ol l(\mu))&=\ol\pt^*\! (\nabla_\ol l \mu) \\
& =(-1)^q \sum g^{\ol \delta \gamma}\nabla_\gamma\nabla_\ol l \mu_{\ol \beta_1, \ldots, \ol \beta_{q-1},\ol \delta}\, dz^{\ol \beta_1}\we \ldots \we dz^{\ol \beta_{q-1}}\\
&=(-1)^{q+1}\sum g^{\ol \delta \gamma} R^h_{\gamma\ol l}(\mu_{\ol\beta_1,\ldots,\ol \beta_{q-1},\ol \delta})\, dz^{\ol \beta_1}\we \ldots \we dz^{\ol \beta_{q-1}}\\
&= - \sum g^{\ol \delta \gamma}R^h_{\gamma\ol l} (\mu_{\ol \delta, \ol \beta_1,\ldots, \ol \beta_{q-1}}) dz^{\ol \beta_1}\we\ldots\we dz^{\ol \beta_{q-1}}
\end{align*}
so that the claim follows with \eqref{eq:cap}.
\end{proof}

\begin{lemma}\label{lemma:property1ofCTEconstruction}
Let $\mu$ be a $\ol\pt$-closed $(0,q)$-form with values in $F$. Then on all fibers $X\times \{s\}$ the following identity holds.
\[(\nabla_k \nabla_{\overline{l}}\,\mu) =\ol{\partial}(\nabla_k F_{\overline{l}}(\mu)) -(\rho_k \cup F_{\overline{l}}(\mu)).
\]
\end{lemma}
\begin{proof}
We pick up the statement of Lemma \ref{lemma:covarianttoexterior} and compute
$\nabla_k \ol\pt F_\ol l(\mu)$. The claim follows from \eqref{eq:curvhend}.
\end{proof}

Now we will finish the proof of Theorem~\ref{thm:maintheorem} by treating the summand $S_2$. The first step is to use \eqref{curvature:liebracketcovariantderiv} and split $S_2$ into two parts. For the inner product taken at $s_0$ we have
\begin{align*}
  \nonumber
  S_2
  &=
  \langle
    \nabla_{\overline{l}} \nabla_k \xi_\rho, \xi_\sigma
  \rangle
  =
  \langle
    \nabla_k \nabla_{\overline{l}} \xi_\rho, \xi_\sigma
  \rangle
  -
  \langle
    \rho_{k\overline{l}} \cup \xi_\rho, \xi_\sigma
  \rangle \\
  &=:
  S_{2a} + S_{2b}\, .
\end{align*}
For $S_{2a}$  we get from Lemma \ref{lemma:property1ofCTEconstruction} that
\begin{align*}
  S_{2a}
  &=
  \langle
    \ol{\partial}( \nabla_k F_{\overline{l}}(\xi_\rho)), \xi_\sigma\rangle -\langle \rho_k\cup F_{\overline{l}}(\xi_\rho), \xi_\sigma \rangle  \\
      &=
  \big\langle (\nabla_k F_{\overline{l}}(\xi_\rho)), \ol{\partial}^* \xi_\sigma \big\rangle - \langle F_{\overline{l}}(\xi_\rho), \rho^*_k \cap \xi_\sigma \rangle
  \end{align*}
holds on $X\times \{s_0\}$.
By assumption the forms $\xi_\rho$ are fiberwise harmonic: $\ol{\partial}^*\xi_\sigma = 0$. We apply Lemma~\ref{lemma:property2ofCTEconstruction}, and again at $s_0$ we obtain
\begin{align*}
  S_{2a}&= - \langle  F_{\overline{l}}(\xi_\rho),\ol{\partial}^* \ol{\partial}(F_{\overline{k}}(\xi_\sigma))\rangle   =
  - \langle \ol{\partial} F_{\overline{l}}(\xi_\rho), \ol{\partial}(F_{\overline{k}}(\xi_\sigma))\rangle .
\end{align*}
We use the fiberwise equation
\[
  \ol{\partial} F_{\overline{l}}(\xi_\rho)  =  \ol{\partial} G\ol{\partial}^* \ol{\partial} F_{\overline{l}}(\xi_\rho)
\]
and, consequently, from Lemma~\ref{lemma:property2ofCTEconstruction} we get:
\begin{align*}
  S_{2a} &=-\big\langle G \ol{\partial}^*\ol{\partial} F_{\overline{l}}(\xi_\rho),\ol{\partial}^*\ol{\partial}(F_{\overline{k}}(\xi_\sigma))\big\rangle =-\big\langle  G(\rho_{l}^* \cap \xi_\rho),\rho_{k}^* \cap\xi_\sigma\big\rangle.
\end{align*}
For the last summand $S_{2b}$ we apply Lemma \ref{lemma:laplacerhokl} to obtain fiberwise
\begin{align*}
  \rho_{k\overline{l}}&=
  H(\rho_{k\overline{l}}) + G (\sqrt{-1} \Lambda_g \left[\rho_k,\rho_{\overline{l}}\right]).
\end{align*}
Hence, again fiberwise
\begin{align*}
  S_{2b}
  &=
  -
  \langle
      \rho_{k\overline{l}}
    \cup
      \xi_\rho,
      \xi_\sigma
  \rangle  \\
  &=
  -
  \langle
    G
    \left(
      \sqrt{-1}
      \Lambda_g
      \left[
          \rho_k,
          \rho_{\overline{l}}
      \right]
    \right)
    \cdot
      \xi_\rho,
      \xi_\sigma
  \rangle
  % \\
%   &\hphantom{=}\,\,-
  -
  \langle
    H
    \left(
        \rho_{k\overline{l}}
    \right)
    \cdot
      \xi_\rho,
      \xi_\sigma
  \rangle.
\end{align*}
This concludes the proof of Theorem \ref{thm:maintheorem}. \qed

In the sequel we discuss the above term  $H(\rho_{k\overline{l}})$. Note that $H(\rho_{k\overline{l}})(s)$ is a harmonic section of ${\rm End}(F_s)$ i.e.\ a constant multiple of the identity, since the holomorphic vector bundles $F_s$ are simple by assumption. So $H(\rho_{k\overline{l}})$ can be identified with a differentiable function on $S$. This argument yields the following lemma.
% ********************************************************************************************************************************************************
% Families with fixed determinant
% ********************************************************************************************************************************************************
% \section{Families with fixed determinant}\label{sec:detfix}
\begin{lemma}\label{le:Hkl}
Let $r= {\rm rk}(F)$. Then the ${\rm End}(F_s)$-valued harmonic sections $H(\rho_{k\overline{l}})(s)$ are of the form
$\Phi_{k\ol l}(s) \cdot {\rm id}_{F_s}$ satisfying the equation
\begin{equation}\label{eq:Hkl}
\Phi_{k\ol l}(s) = \frac{1}{r}  \int_{X \times\{s\}} {\rm tr}(R^h_{k \ol l}) \, \omega^n\Big/\int_X \omega^n.
\end{equation}
\end{lemma}

Since Hermite-Einstein metrics are only unique up to a constant positive factor, the given metrics $h$ on $F$ over $X\times S$ may be rescaled, i.e.\ modified by a factor ${\rm exp}(\varphi(s))$, where $\varphi$ denotes a differentiable function:

\begin{proposition}
Let $(F,h) \to X \times S$ be a holomorphic, hermitian vector bundle such that the restrictions $(F_s,h_s)$ are simple Hermite-Einstein vector bundles over a compact Kähler manifold $(X,g)$ such that all ${\rm det}(F_s)$ are isomorphic to a fixed line bundle $L$ on $X$.

Then locally with respect to $S$ the metric $h$ can be rescaled such that all harmonic projections $H(\rho_{k \ol l})(s)$ vanish.
\end{proposition}

\begin{proof}
Let $q:X\times S \to X$ be the canonical projection. We may assume that $\det F \otimes q^*L^{-1}$ is trivial. We equip $L$ with an auxiliary hermitian metric $h_L$ and consider $\det h \cdot q^*(h_L^{-1})$, which is of the form ${\rm exp}(\chi)$. Now
$$
{\rm tr}(R^h_{k\ol l})= -\frac{\pt^2 \log \det h}{\pt s^k\pt s^{\ol l}} = -\frac{\pt^2 \chi}{\pt s^k\pt s^{\ol l}},
$$
since the extra additive term involving $\log h_L$ does not depend upon $s$. Now the components $\Phi_{k\ol l}(s)$ in the sense of \eqref{eq:Hkl} are of the form
$$
\Phi_{k\ol l}(s)= -\frac{\pt^2 \varphi(s)}{\pt s^k \pt s^\ol l},
$$
where
$$
\varphi(s)= \frac{1}{r}\int_{X\times \{s\}}\chi\, \omega^n \Big/ \int_X \omega^n.
$$
We replace $h$ by ${\rm exp }(-\varphi)h$, which yields the claim.
\end{proof}
If $S$ is an arbitrary base space, the bundle $\det F \otimes q^*L^{-1}$ is of the form $p^*M$, where $M$ is a holomorphic line bundle on $S$. With the same methods one can see immediately the following somewhat more general fact.
\begin{proposition}
For any hermitian metric $h_M$ on $M$, there is a hermitian metric $h$ on $F$ that restricts to a family of Hermite-Einstein metrics $h_s$ such that
$$
\sqrt{-1}H(\rho_{k \ol l})ds^k \we ds^\ol l = - \sqrt{-1} \pt \ol\pt \log h_M
$$
on $S$.
\end{proposition}
In a moduli theoretic situation of a family with $\det(F_s)\simeq L$ for all $s$, after replacing $S$ by a finite unbranched covering $\pi:\wt S \to S$, there exists a line bundle $\wt M$ on $\wt S$ such that $\wt M^{\otimes r} = \pi^* M$. Now the bundle $F$ can be replaced by $\wt F= \pi^* F \otimes \wt q^*\wt M^{-1}$ where $\wt q:X\times \wt S \to \wt S$ is the canonical projection.  The isomorphism classes of the fibers are unchanged.
\begin{corollary}
In the above situation, the bundle $\wt F$ possesses a family of Hermite-Einstein metrics such that everywhere all $H(\rho_{k \ol l})(s)$ vanish.
\end{corollary}
Let a family $(F,h)$ of (simple) holomorphic Hermite-Einstein bundles be given. We have the induced family of Hermite-Einstein metrics $\widetilde h_s$ on ${\rm End}(F_s)$, and the connection and curvature forms on ${\rm End}(F_s)$ are given by $[\theta^h_s, \,\,]$ and $[\Omega^h_s, \,\,]$ resp.\ so that ${\rm tr}(R^{\widetilde h}_{k \ol l})=0$. This fact implies that the induced term for the curvature formula for the direct images vanishes. Hence the following proposition holds.

\begin{proposition}
Given a holomorphic family of simple, holomorphic Hermite-Einstein bundles, equip the endomorphism bundles with the induced structure.  Then the fiberwise harmonic projections $H(\rho_{k \ol l})$ for the sheaves $R^qp_*({\rm End}(F))$ vanish identically.
\end{proposition}

% ********************************************************************************************************************************************************
% Bibliography
% ********************************************************************************************************************************************************


\begin{thebibliography}{TW98}

\bibitem[Be09]{Be09} {\scshape B.~Berndtsson}: {\em Curvature of vector bundles associated to holomorphic fibrations}, Ann.\ of Math., Vol.\ 169, 531--560 (2009).

\bibitem[BP08]{BP08} {\scshape B.~Berndtsson, M.~P\u{a}un}: {\em Bergman kernels and the pseudoeffectivity of relative canonical bundles}, Duke Math.\ J., Vol.\ 145, No.\ 2, 341--378 (2008).

\bibitem[BS76]{BS76} {\scshape C.~B\u{a}nic\u{a}, O.~St\u{a}n\u{a}\c{s}il\u{a}}: {\em Algebraic Methods in the Global Theory of Complex Spaces}, John Wiley \& Sons (1976).

\bibitem[KS58]{KS58} {\scshape K.~Kodaira, D.~C.~Spencer}: {\em On deformations of complex analytic structures, I.}, Ann.\ of Math., Vol.\ 67, No.\ 2, 328--401 (1958).

\bibitem[KS60]{KS60} {\scshape K.~Kodaira, D.~C.~Spencer}: {\em On deformations of complex analytic structures, III. Stability theorems for complex structures}, Ann.\ of Math., Vol.\ 71, No.\ 1, 43--76 (1960).

\bibitem[MT08]{MT08} {\scshape C.~Mourougane, S.~Takayama}: {\em Hodge metrics and the curvature of higher direct images}, Ann.\ Sci.\ \'Ec.\ Norm.\ Sup\'er., Vol.\ 41, 905--924 (2008).

\bibitem[MT09]{MT09} {\scshape C.~Mourougane, S.~Takayama}: {\em Extension of twisted Hodge metrics for Kähler morphisms}. J.~Differ.~Geom.\ Vol.\ 83, No.\ 1, 131--161 (2009).

\bibitem[Ov92]{Ov92} {\scshape M.~Overhaus}: {\em Die Petersson-Weil-Metrik auf dem Modulraum der Hermite-Einstein-B\"undel}, Dissertation, Bochum (1992).

\bibitem[Sch12]{Sch12} {\scshape G.~Schumacher}: {\em Positivity of relative canonical bundles and applications}, Invent.\ Math., Vol.\ 190, 1--56 (2012).

\bibitem[ST92]{ST92} {\scshape G.~Schumacher, M.~Toma}: {\em On the Petersson-Weil metric for the moduli space of Hermite-Einstein bundles and its curvature}, Math.\ Ann., Vol.\ 293, 101--107 (1992).

\bibitem[TW98]{TW98} {\scshape W.-K.~To, L.~Weng}: {\em Curvature of the $L^2$-metric on the direct image of a family of Hermitian-Einstein vector bundles}, Am.\ J.\ Math., Vol.~120, No.~3, 649--661 (1998).


\end{thebibliography}
\end{document}